\documentclass[12pt]{article}
\usepackage[english]{babel}
\usepackage{amsfonts,amsmath,amsxtra,amsthm,latexsym}
\usepackage{amssymb} 

\newtheorem{thm}{Theorem}[section]
 \newtheorem{cor}[thm]{Corollary}
 \newtheorem{lem}[thm]{Lemma}
 \newtheorem{prop}[thm]{Proposition}
 \theoremstyle{definition}
 \newtheorem{defn}{Definition}[section]
 \theoremstyle{remark}
 \newtheorem{rem}{Remark}[section]
 \numberwithin{equation}{section}

\DeclareMathOperator{\im}{Im}
\DeclareMathOperator{\re}{Re}

\DeclareMathOperator{\supp}{supp}

\def\R{\mathbb R}
\def\C{\mathbb C}
\def\N{\mathbb N}

\def\D{\mathbb D}
\def\S{\mathfrak{S}}
\def\T{\mathbb T}

\def\vphi{\varphi}

\def\i{\mathrm{i}}

\def\I{\mathcal{I}}

\def\fin{\mathrm{fin}}
\def\Ad{Ad}

\def\Tm{\mathrm{Tm}}

\def\d{\mathrm{d}}
\def\k{\kappa}
\def\pa{\partial}
\def\l{\ell}

\def\Stm{\mathrm{Stm}}

\begin{document}
\title{Optimization of quasi-normal eigenvalues for
Krein-Nudelman strings}
\author{}
\date{}
\maketitle
{\center {\large Illya M. Karabash } $^{\text{a,*}}$\\[4mm]

{\small $^{\text{a,*}}$  Institute of Applied Mathematics and Mechanics of NAS of Ukraine,

 R. Luxemburg str. 74, Donetsk, 83114, Ukraine.

E-mail addresses: i.m.karabash@gmail.com, karabashi@mail.ru

$^*$Corresponding author

}}

\begin{abstract}
The paper is devoted to optimization of resonances for Krein strings
with total mass and statical moment constraints. The problem is to
design for a given $\alpha \in \R$ a string that has a resonance on
the line $\alpha + \i \R$ with a minimal possible modulus of the
imaginary part. We find optimal resonances and strings explicitly.
\end{abstract}

\quad\\
MSC-classes: 49R05, 78M50, 35P25, 47N50

\section{Introduction}

Recently the increasing interest in loss mechanisms of structured optical and mechanical systems has
 given rise to spectral optimization problems for dissipative models involving wave equations in inhomogeneous media,
 see e.g. \cite{AASN03} and references therein.
The question is how to design a system with very high or very low
loss of energy for oscillations with frequencies in a given range. The rate of energy decay
is closely connected to imaginary parts of eigenvalues of the
corresponding non-self-adjoint operator. In the paper, these
eigenvalues are called quasi-(normal) eigenvalues. Naively, the
closer quasi-eigenvalues to the real axis $\R$, the less the rate of
energy decay.

It seems that the systematic study of eigenvalue's maximization and
minimization problems associated with self-adjoint elliptic
operators was initiated by M.G.~Krein \cite{K51}. While there exists an extensive
literature on spectral optimization associated with self-adjoint
elliptic operators, an analytic background for spectral optimization
problems involving non-self-adjoint operators is not well developed.
A possible explanation for this fact is that, for self-adjoint
problems, eigenvalues move on the real line and do not have root
eigenfunctions of higher order. This leads to a relatively simple
statement of the optimization problem and to a relatively simple
perturbation theory. Quasi-eigenvalues' behavior is much more
complex.

The goal of the present paper is to study quasi-eigenvalue
optimization problems analytically for Krein strings with
dissipation at one end.

In the settings of this paper, \emph{a (Krein) string} is the
interval $(-\infty, 1]$ carrying a dispersed mass, which is
represented by a locally bounded nonnegative Borel measure $\d M$.
We will speak about the string $\d M$ because the string is
completely determined by the measure. If the closed support $\supp
\d M$ of the measure is finite, the string is called \emph{regular}
and, by definition, \emph{the left end} $a_1$ of a string  is the
left end of $\supp \d M$. In the trivial case when $\d M $ is the zero measure, we put $a_1 =1$. We denote \emph{the class of regular
strings} by $\S_\fin$.

The quasi-eigenvalue problem for a regular string
 \begin{eqnarray} \label{e dMdx eq Int}
- \frac{\d^2}{\d M \d x} y(x)= \k^2 y(x) , \ \ \pa_x^- y (a_1) = 0,
\ \ y(1) = \frac{i}{\k} \, \pa_x^+ y(1)
\end{eqnarray}
is considered. Here the Krein-Feller differential expression
$\frac{\d^2}{\d M \d x} $ can be understood in the integral sense, and $ \pa_x^{+(-)}
$ is a properly defined right-hand (resp., left-hand) derivative
(see e.g. \cite{KK68_II,A75,DM76} and Section \ref{s sol} for
details). Problem (\ref{e dMdx eq Int}) corresponds to free
transverse harmonic oscillations of a string with the left end
sliding without friction and the right end $x=1$ with friction
proportional to the velocity of motion.

Eigen-parameters $\k \in \C$ such that (\ref{e dMdx eq Int}) has a
nonzero solution $y $ will be called \emph{quasi-eigenvalues}.
Corresponding eigenfunctions $y$ are called \emph{quasi-normal
modes}. The real part $\alpha=\re \k $ of the quasi-eigenvalue is the
\emph{frequency} of oscillations corresponding to the quasi-normal
mode $y$, the imaginary part $\beta = \im \k$ is always positive and
characterizes the \emph{rate of decay} of the oscillations.

Several other names for $\k$ are used, sometimes in slightly
different settings: dissipation frequencies \cite{KN79,KN89},
resonances and quasi-normal levels (in connection with the
time-independent Schr\"{o}dinger equation and in the Physics
literature).

The set of quasi-eigenvalues $\k$ of the string $\d M$ is denoted by
$K (\d M) $. It is known that $K (\d M) \subset \C_+$, that
quasi-eigenvalues are isolated, and that $\infty$ is their only
possible accumulation point (see e.g. \cite{KN79,KN89} and Section
\ref{s tools}).

Constraining for regular strings \emph{the total mass}
\[
\Tm_M :=
\int_{(-\infty,1]} \d M
\]
and \emph{the statical moment} with respect to (w.r.t.) the right end
\[
\Stm_M := \int_{-\infty}^1 (1-x) \d M ,
\]
we define for $m,S>0$ the admissible  family
\begin{eqnarray} \label{e wtAdmes}
\Ad = \Ad (m,S) :=  \{ \d M \in \S_\fin \ : \ 0 < \Tm_M \leq m , \ 0
\leq \Stm_M \leq S \} .
\end{eqnarray}

\emph{The optimization problem is}
\begin{itemize}
\item[(i)] for $\alpha \in \R$, to find
\begin{equation} \label{e Ialpha}
\I (\alpha) \ := \ \inf \{ \im \k \ : \ \re \k = \alpha \text{ and } \k \in K (\d M)
\text{ for certain } \d M \in \Ad \},
\end{equation}
\item[(ii)]
to find all the strings $\d M$ such that $\alpha + \i \I (\alpha) \in K (\d M)$, or,
if such a string does not exist, to find a sequence of strings
$\{ \d M^{(n)} \}_{n=1}^\infty \subset \Ad $ such that $\alpha + \i \beta_n \in K (\d M^{(n)})$
with $\beta_n \to \I (\alpha)+0$ as $n \to \infty$.
\end{itemize}

In the case when there exists $\d M \in \Ad$ such that $\alpha + \i \I (\alpha) \in K (\d M)$
(i.e., when the minimum in (\ref{e Ialpha}) is achieved), we call $\k^{[\alpha]} := \alpha + \i \I (\alpha) $
\emph{optimal quasi-eigenvalue} and call $\d M$ \emph{optimal string} for the frequency $\alpha$.

In Section \ref{s sol}, we solve this optimization problem.  It occurs that for $\alpha \not \in ( -S^{-1/2} , S^{-1/2} )$
 the infimum in (\ref{e Ialpha}) is not achieved and optimal strings do not exist.
For $\alpha \in ( -S^{-1/2} , S^{-1/2} )$ the optimal strings exist and we find them and corresponding optimal quasi-eigenvalues
explicitly. It occurs that optimal strings $\d M$ consist of a single atom mass placed such that
one of the equalities $\Tm_M = m$ and $ \Stm_M = S$ hold, see Theorems \ref{t
prop of optim}. In other words, optimal strings $\d M$
are extreme points of $\Ad$. Section \ref{s tools} is preparative for
the proofs of these results. The proofs are given in Section \ref{s PrTh mes}.


We use the fact that for the class of strings with finite $\Tm_M$ and $\Stm_M$ a complete solution
of the corresponding direct spectral problem was obtained by Krein and Nudelman \cite{KN89}.
While there exist a number of papers on the direct and inverse
spectral problems for quasi-eigenvalues (see e.g. \cite{CZ95,GP97,Sh07}),
it seems that certain strong additional conditions on $\d M$ and the friction coefficient are always involved.
In the author's opinion,  the study of the quasi-eigenvalue direct spectral problem for
absolutely continuous $\d M$ with densities  $M'(x)$ in $L^p$-spaces could help to understand better
the related optimization problems.


\textbf{Notation}. $\C_\pm = \{ z \in \C : \pm \im z >0 \}$, \ $\R_\pm = \{ x \in \R: \pm x >0 \}$,
 $\D_\epsilon (\zeta) := \{z \in \C : | z - \zeta | < \epsilon \}$,
$\T_\epsilon (\zeta) = \{ z \in \C : | z - \zeta | = \epsilon \}$.
For $\Omega \subset \C$, $v_0,z \in \C$, let $z\Omega +
v_0 := \{ zv + v_0 \, : \, v \in \Omega \}$.



\section{Optimal strings and quasi-eigenvalues.}
\label{s sol}

Following the settings of
\cite{KN89} with some minor changes (see also \cite[Ch.~5]{DM76}), we consider a finite or semi-infinite (Krein) string $\d M$ on $(-\infty,1]$
with a finite \emph{statical moment}   (first moment) w.r.t. the right end $x=1$, i.e., $ \Stm_M =  \int_{(-\infty,1]} (1-x) \d M < \infty $.
\emph{This class of strings is denoted by} $\S$.
Then \emph{the total mass} of the string is finite, $\Tm_M := \int_{(-\infty,1]} \d M < \infty$.
Through the standard procedure the nondecreasing function $M(x) := \int_{(-\infty,x]} dM $ can be associated with the Borel measure $\d M$.
By definition, \emph{the left end of the string} is
\begin{equation} \label{e a1}
 a_1:= \inf \{ x \ : \ M(x) > 0  \} \ \ (\geq - \infty)  .
\end{equation}
When $M (x)  =0 $ on $(-\infty,1]$, we put $a_1 =1$.
Strings with finite $a_1$ are called \emph{regular}, with $a_1 = -\infty$ singular. \emph{The
class of regular strings is denoted by} $\S_\fin$.

The quasi-eigenvalue problem for strings is given by (\ref{e dMdx eq Int}).
It was noticed in \cite{KN79,KN89} that it is convenient to include the case $a_1=-\infty$ (i.e., singular strings)
into the study of problem (\ref{e dMdx eq Int}) defining the left-hand derivative  at $-\infty$ by
$\pa_x^- y (-\infty) := \lim_{x \to -\infty} \pa_x^- y (x) $.
To define the right-hand derivative $ \pa_x^+ y (1)$, we assume that
the function $M $ is continued to $x \in (1,+\infty) $ by $ M(x):= M(1)+(x-1)$
(i.e., $\d M$ is continued by the Lebesgue measure)
and that $y$ satisfies $ \frac{\d^2}{\d M \d x} y(x) + \k^2 y(x) = 0$ in a vicinity of $x=1$.
Taking this into
account, put
\begin{equation} \label{e a2}
 a_2:= \inf \{ x \in \R \ : \ \d M(s) = \d s \text{ on } [x,+\infty) \}.
\end{equation}
Clearly, $ -\infty < a_2 \leq 1 $ and $a_1 \leq a_2$. If $a_2 =1$,
the string $\d M$ is called \emph{reduced}.

Considering the problem in the weighted Hilbert space $L^2 (-\infty,
1; \d M)$ with the norm $ \| f \|_{L^2 (\d M)} = \left(
\int_{(-\infty,1]} |f|^2 \d M \right)^{1/2}$ (i.e., assuming $y$ and
$\frac{\d^2}{\d M \d x} y$ in $L^2 (-\infty, 1; \d M)$), we call the
eigen-parameters $\k$ corresponding to nonzero solutions of (\ref{e
dMdx eq Int}) \emph{quasi-eigenvalues of} $\d M$ and denote
\emph{the set of quasi-eigenvalues} by $K (\d M)$.

It is not difficult to prove that $ K (\d M) \subset \C_+ $ for all $\d M \in \S$.
It occurs  that $K (\d M)$ is the set of zeroes of the entire function
\begin{equation*} \label{e F}
F(z) = F (z; \d M) := \vphi (1,z) - \i \pa_x^+ \vphi (1,z) \, / \, z , \ \ z \in \C,
\end{equation*}
where $\vphi (x,z) = \vphi (x, z; \d M )$ is the solution of the initial value problem
\begin{equation*} \label{e phi}
\frac{\d^2}{\d M \d x} \vphi (x,z) = - z^2 \vphi (x,z) , \ \ \ \vphi (a_1,z) = 1, \ \ \ \pa_x^- \vphi (a_1,z) = 0 .
\end{equation*}
In the case $a_1=-\infty$, it is assumed that $\vphi (-\infty,z) :=
\lim_{x\to -\infty} \vphi (x,z)$ and the existence of $\vphi $ follows
from the theory of Krein strings  (see \cite{KK68_II,KN79,KN89} for details and references).

It is obvious that all modes $y$ corresponding to
$\k \in K (\d M)$ are equal to $\vphi (\cdot, \k; \d M )$ up to a multiplication by a constant.  So \emph{the geometric multiplicity} of any quasi-eigenvalue equals 1. In the following, \emph{the multiplicity} of a quasi-eigenvalue means its \emph{algebraic multiplicity}.

 \begin{defn}[\cite{KN79,KN89}] \label{d mult}
\emph{The multiplicity} of a quasi-eigenvalue is its multiplicity as a zero of the entire function $F (\cdot) $. A quasi-eigenvalue is called \emph{simple} if its multiplicity is $1$.
 \end{defn}

For regular strings, this is classical M.V. Keldysh's definition of multiplicity for eigenvalue problems with an eigen-parameter in boundary conditions (see e.g. \cite[Sec. 1.2.2-3]{N69}).

By $K_r (\d M)$ we denote the set of quasi-eigenvalues of multiplicity $r \in \N$. Each quasi-eigenvalue has a finite multiplicity.

Let us introduce the set
$ K (\Ad ) := \bigcup\limits_{\d M \in \Ad } K (\d M ) $
of all possible quasi-eigenvalues for strings from  $\Ad$.
Recall that the admissible family $\Ad$ defined by (\ref{e wtAdmes}).
Then the function $\I$ given by (\ref{e Ialpha}) can be written in the form
\begin{equation*} 
\I (\alpha) \ = \ \inf \{ \im \kappa \ : \ \re \kappa = \alpha \text{ and } \k \in K (\Ad) \} .
\end{equation*}


\begin{thm} \label{t K mes}
$ K (\Ad ) = \i [ m^{-1}, + \infty ) \cup \Bigl( \C_+ \setminus \left[ \D_{m^{-1}} ( \i m^{-1}) \cup \D_{S^{-1/2}} (0) \right] \Bigr) . $
\end{thm}

We postpone the proof of this result to Section \ref{s PrTh mes}  and turn to its immediate corollary.

\begin{cor} \label{c I mes}
If $S \leq m^2 /4$, then
$
 \I (\alpha) = \left\{ \begin{array}{ll}
m^{-1}, & \text{ for } \alpha = 0 \\
\sqrt{S^{-1} - \alpha^2} , & \text{ for } 0< |\alpha| \leq S^{-1/2}  \\
0, & \text{ for } \alpha \geq S^{-1/2}
\end{array}
\right. .
$

If $S > m^2 /4$, then
$
 \I (\alpha) = \left\{ \begin{array}{ll}
m^{-1}, & \text{ for } \alpha = 0 \\
m^{-1} + \sqrt{m^{-2} - \alpha^2},  & \text{ for } 0<|\alpha|< \sqrt{S^{-1} - m^2 S^{-2} / 4}   \\
\sqrt{S^{-1} - \alpha^2} , & \text{ for } \sqrt{S^{-1} - m^2 S^{-2} / 4} \leq |\alpha| \leq S^{-1/2}  \\
0, & \text{ for } \alpha \geq S^{-1/2}
\end{array}
\right. .
$
\end{cor}

By $\d \Delta_{x_0, A}$ we denote the string consisting of an atom mass $A>0$ placed at a point $x_0 \leq 1$, i.e., writing with  Dirac's $\delta$-function,
$ \d \Delta_{x_0, A} = A \delta (x-x_0) \d x $.

The following proposition can be obtained by direct calculations.

\begin{prop} \label{p Delt}
Let $x_0 \leq 1 $ and $A > 0$.
Then the sets of quasi-eigenvalues of the strings $\d \Delta_{x_0, A}$ have the following description (taking multiplicities into account):
\item[(i)] For $x_0 < 1 $, $ K (\d \Delta_{x_0, A}) =
\left\{ \ \i \frac {1}{2(1 -x_0 )} \pm \sqrt{ \frac{1}{A (1 -x_0 )} -
\frac {1}{4 (1-x_0)^2} }  \ \right\}.
$
\item[(ii)] In the case $x_0 = 1 $, $K (\d \Delta_{1, A}) = \{ \i A^{-1} \}$.
\end{prop}

\begin{rem} \label{r ex mult}
In the case $4 (1-x_0) =A$, the proposition means that $ \frac {\i}{2(1 -x_0 )}$ is a quasi-eigenvalue
of multiplicity 2.
\end{rem}

In the case $\I (\alpha )=0$, i.e., when $\alpha \in \R \setminus ( -S^{-1/2} , S^{-1/2} )$,  the infimum in (\ref{e Ialpha}) is not achieved.
An optimizing sequence of strings those quasi-eigenvalues tend to $\alpha \in \R \setminus ( -S^{-1/2} , S^{-1/2} )$ is provided by Proposition \ref{p Delt}.
Indeed, the strings $\d \Delta_{1-(2\beta)^{-1}, \, 2\beta (\alpha^2 +\beta^2)^{-1}}$ belong to $\Ad$ for $\beta >0$ small enough
and their sets of quasi-eigenvalues are $\{ \pm \alpha + \i \beta \}$.

On the other side,  $K (\Ad )$ is a closed set in the relative topology of $\C_+$. So for $\alpha \in (- S^{-1/2},S^{-1/2})$ minimizers exist.

\begin{thm} \label{t prop of optim}
Let $\alpha \in (- S^{-1/2},S^{-1/2})$ and let $\k^{[\alpha]} = \alpha + \i \I (\alpha)$
be the corresponding optimal quasi-eigenvalue.
Then for the frequency $\alpha$ there exists a unique optimal string $\d M_\alpha$ (i.e.,  $\d M_\alpha \in \Ad $ and
$\k^{[\alpha]} \in K(\d M_\alpha)$). The string  $\d M_\alpha$ is of the form $\d \Delta_{x_0,A}$ and satisfies at least one of the equalities $\Tm_{M} = m$,
$\Stm_{M} = S$.

More precisely, $\d M_0 = \d \Delta_{1,m} $.

For $0<\alpha < S^{-1/2}$, $\d M_\alpha = \d M_{-\alpha} = \d \Delta_{x_0,A}$ with $x_0$ and $A$ given by the following equalities:
\begin{eqnarray*}
\text{when} & S \leq m^2 /4 ,  &   x_0 = 1 - \frac{1}{2 (S^{-1} - \alpha^2)^{1/2}}, \ \ A = 2S\sqrt{S^{-1} - \alpha^2} ; \\
\text{when}  & S > m^2 /4,  &   (x_0,  A) =
\left\{ \begin{array}{rr}
\left( 1- \frac{m}{2+2(1-\alpha^2m^2)^{1/2}} \, , \,  m \right) ,  &  0<\alpha< \sqrt{\frac{1}{S} - \frac{m^2}{4S^2}}   \\
\left( 1 - \frac{1}{2 (S^{-1} - \alpha^2)^{1/2}} \, , \,
2S\sqrt{S^{-1} - \alpha^2} \right) , & \sqrt{\frac{1}{S} -
\frac{m^2}{4 S^2}} \leq \alpha \leq S^{-1/2}
\end{array} \right.
\end{eqnarray*}
\end{thm}

\begin{rem}
The family $\Ad $ is a convex set in the space of signed Borel
measures on $(-\infty,1]$. Theorem \ref{t prop of optim} immediately
implies that the strings $\d M_\alpha$ are extreme points of
$\Ad$.
\end{rem}

Denote by
$K^{\mathrm{mult}} (\Ad):= \bigcup \, \{ K_r (\d M) \, : \,  r\ge 2 \text{ and } \d M \in \Ad  \}$
the set of all possible non-simple quasi-eigenvalues
 for strings from $\Ad$.  Using $\k^{[\alpha]}$ and $\d M_\alpha$ of
Theorem \ref{t prop of optim}, we state the following.

\begin{thm} \label{t sep mult}
Let $\alpha \in (- S^{-1/2},S^{-1/2})$. Then the optimal quasi-eigenvalue $\k^{[\alpha]}$ is a simple
quasi-eigenvalue of the optimal string $\d M_\alpha$.
Moreover, the distance from $\k^{[\alpha]}$ to $K^{\mathrm{mult}} (\Ad)$ is positive.
\end{thm}

Theorems \ref{t prop of optim} and \ref{t sep mult} will be proved in Section \ref{s PrTh mes}.

\begin{rem}
All the results of this subsection are valid if singular strings are included into the admissible family,
i.e., if one replaces $\S_\fin$ with $\S$ in (\ref{e wtAdmes}).
The proofs do not require changes.
\end{rem}

\section{Tools. \label{s tools}}


 Denote
$ \l := 1 - a_2 , $
  where $a_2$ is defined by (\ref{e a2}). So $\l$ is \emph{the length of a maximal interval $(x,1]$ carrying Lebesgue's measure}, $\l = 0$
  exactly when the string is reduced.
  It is easy to see that
 $K (\d M (x)) = K (\d M (x-\l))$ (see e.g. \cite[Sec. 3.1]{KN89});
  that is, if we delete the interval $(a_2, 1]$  from a non-reduced string and shift the new right endpoint to $1$,
  the obtained reduced string has the same quasi-eigenvalues.

  All possible sets of quasi-eigenvalues for strings of the class $\S$ were characterized in \cite{KN79,KN89}.

 \begin{thm}[Theorem 3.1 of \cite{KN89}] \label{t KN}
 Let $ \{\k_j \}$ be an (empty, finite, or infinite) sequence of complex numbers, some of those may coincide. Then
 the set $\{\k_j \}$ is the set of quasi-eigenvalues (taking multiplicities into account) of a certain string  $\d M \in \S$
 if and only if all the following conditions are fulfilled:
 \item[1)] The set $ \{\k_j \} $ is symmetric w.r.t.
 the imaginary axis $\i \R$, moreover, the multiplicities of symmetric numbers are the same.
 \item[2)] $\im \k_j > 0$ for all $j$, and $\sum_j |\im (1/\k_j)| < \infty$.
 \item[3)] $\sum_j |\k_j|^{-2} < \infty$.

If these conditions are fulfilled, there exists a unique reduced string $\d M$ with the set of quasi-eigenvalues
$\{\k_j \}$.
\end{thm}

According to the last equality in \cite[Sec.4.1]{KN89},
 \begin{equation} \label{e sum im}
 \sum_j | \im (1/\k_j) | = \Tm_M - \l .
 \end{equation}

We use the following power series decomposition of  $\vphi$ (see e.g. \cite[Sec. 1.1]{KN89})
\begin{eqnarray}
\vphi (x , \k ) = 1 - \vphi_1 (x) \k^2 + \vphi_2 (x) \k^4 - \vphi_3 (x) \k^6 + \dots , \label{e pow ser}\\
\vphi_0 (x) \equiv 1 , \ \  \vphi_{j} (x) = \int_{-\infty}^x (x - s)
\vphi_{j-1} (s) \d M (s) , \ \ j \in \N , \label{e phij}
\end{eqnarray}
 to get a formula for the statical moment in terms of quasi-eigenvalues' positions, $\l$,  and $\Tm_M$.

\begin{prop} \label{p Stm=}
\begin{eqnarray}
\Stm_M & = & \sum_{\re k_j >0} \frac{1}{|\k_j|^{2}} +
4 \sum_{\substack{\re k_j, \re k_n >0\\ j \neq n}}
\frac{\im \k_j \im \k_n }{|\k_j|^2 |\k_n|^2 } +
2 \sum_{\substack{\re k_j > 0 \\ \re k_n = 0}}
\frac{\im \k_j \im \k_n }{|\k_j|^2 |\k_n |^2} \notag \\
& + &
\sum_{\substack{\re k_j = \re k_n = 0 \\ j \neq n }} \frac{\im \k_j \im \k_n }{|\k_j|^2 |\k_n|^2 } + \l \left( \Tm_M - \frac{\l}{2} \right).  \label{e Stm=}
\end{eqnarray}
\end{prop}

\begin{proof}
In our notation, \cite[formula (3.9)]{KN89}
takes the form
\[
\vphi (1,\k) - \i \pa_x^+ \vphi (1,\k) / \k =
e^{\i \k \l }
\prod_{\re k_j >0} \left[ \left( 1 - \frac{\k}{\k_j}\right) \left( 1 + \frac{\k}{\overline{\k_j}} \right) \right]
\prod_{\re k_j =0} \left( 1 - \frac{\k}{\k_j}\right) .
\]
(Note that our $\vphi (1,\k)$  is $\vphi (b,\k^2)$ of  \cite{KN89} and that the formula for $Q$ in \cite{KN89}
is given with a misprint, cf. \cite{KN81} and also \cite[pp. 303-6]{K93}).

Plugging (\ref{e pow ser}) into the left side,
\begin{eqnarray*}
& \left[ 1 - \k ^2 \vphi_1 (1) + O(\k^4) \right] -
\frac{\i}{\k } \left[ - \k^2 \pa_x^+ \vphi_1 (1) + O(\k^4) \right] = \\
& \quad \left[ 1 + \i \k \l - \k^2 \l^2 /2 + O(\k^3) \right]
\prod_{\re k_j >0} \Bigl[ 1 - \k \left( \k_j^{-1} -\overline{\k_j^{-1}} \right) - \k^2 |\k_j|^{-2} \Bigr]
\prod_{\re k_j =0} \left( 1 - \k/\k_j \right) ,
\end{eqnarray*}
and
comparing the coefficient of $\k^2$,
we get
\begin{eqnarray}
- \vphi_1 (1) & = & - \sum_{\re k_j >0} \frac{1}{|\k_j|^{2}} -
4 \sum_{\substack{\re k_j, \re k_n >0\\ j \neq n}}
\frac{\im \k_j \im \k_n }{|\k_j|^2 |\k_n|^2 } -
2 \sum_{\substack{\re k_j > 0 \\ \re k_n = 0}}
\frac{\im \k_j \im \k_n }{|\k_j|^2 |\k_n |^2}
\notag \\
& - &
\sum_{\substack{\re k_j = \re k_n  = 0 \\ j \neq n }} \frac{\im \k_j \im \k_n }{|\k_j|^2 |\k_n|^2 } - \frac{\l^2}{2} - \l \sum_{j} \frac{ \im \k_j }{|\k_j|^2 } \label{e -phi1}
\end{eqnarray}
(the symmetry of $K (\d M)$ w.r.t. $\i \R$ was used to modify the last term).
By (\ref{e phij}), $\vphi_1 (1) = \Stm_M $.
By (\ref{e sum im}), the last term in (\ref{e -phi1})
equals $\l (\Tm_M - \l)$. Taking into account these equalities, one can write (\ref{e -phi1}) as (\ref{e Stm=}).
\end{proof}

\begin{rem}
 Squaring (\ref{e sum im}) and taking again into account the symmetry of $K (\d M)$ w.r.t. $\i \R$, it is not difficult to notice that (\ref{e Stm=}) can be written shorter:
\[
\Stm_M = \frac {1}{2} \Tm_M^2 + \frac {1}{2}
\sum_j \frac{(\re \k_j)^2 - (\im \k_j)^2 }{|\k_j|^4 } .
\]
However (\ref{e Stm=}) is more convenient for the needs
of the next section.
\end{rem}

\section{Proofs of Theorems \ref{t K mes}, \ref{t prop of optim}, and \ref{t sep mult}.}
\label{s PrTh mes}

\begin{lem} \label{l ab in Kr a>0}
 Let $ \k = \alpha + \i \beta \in K (\d M)$ and $ \alpha \neq
 0$.Then the following assertions hold.
\item[(i)]  $\k \not \in \D_{\Tm_M^{-1}} \left(\i \, \Tm_M^{-1} \right)$ and $\k \not \in \D_{\Stm_M^{-1/2}} (0)$.
\item[(ii)] If $\k \in \T_{\Tm_M^{-1}} \left(\i \, \Tm_M^{-1} \right)$, then $K(\d M)$
consists of two simple quasi-eigenvalues $\pm\alpha + \i \beta$ and \\
$\d M = \d \Delta_{1-(2\beta)^{-1}, \,\Tm_M} $.
\item[(iii)] If $\k \in \T_{\Stm_M^{-1/2}}
(0)$, then $K(\d M)$ consists of two simple quasi-eigenvalues
$\pm\alpha + \i \beta$ and \\ $\d M = \d \Delta_{x_0,A} $ with
$\displaystyle (x_0, A) = \left( 1 - \frac{1}{2 (\Stm_M^{-1} - \alpha^2)^{1/2}} , 2\Stm_M\sqrt{\Stm_M^{-1} -
\alpha^2} \right) $.
\item[(iv)] If $ \k \in K_r (\d M)$ (i.e., the multiplicity of
$\k$ is $r$), then $\k \not \in \D_{r \,\Tm_M^{-1}} \left(\i r \,
\Tm_M^{-1} \right)$ and $\k \not \in \D_{\sqrt{r\, \Stm_M^{-1}}} \,
(0)$.
\end{lem}

\begin{proof}
If $ \k \in K (\d M)$, formula (\ref{e sum im}) and Theorem \ref{t KN} (1)
imply $ \frac{2 \beta}{\alpha^2 + \beta^2} \leq \Tm_M $. Theorem
\ref{t KN} (1) and (\ref{e Stm=}) imply $(\alpha^2 + \beta^2)^{-1}
\leq \Stm_M $ (note that the definition of $\l$ yields $\l \le
\Tm_M$). These inequalities are equivalent to \textbf{(i)}. Note
that the equalities hold exactly when $K(\d M) = \{ \pm\alpha + \i
\beta \}$ and $\l=0$. Combining this, Proposition \ref{p Delt}, and
the uniqueness statement of Theorem \ref{t KN} one can easily get
\textbf{(ii)} and \textbf{(iii)}. In the case when $\k \in K_r (\d
M)$, (\ref{e sum im}) and (\ref{e Stm=}) yield \textbf{(iv)}.
\end{proof}

The case $\k \in \i \R$ is simpler, and the above arguments lead to the following lemma.

\begin{lem} \label{l ab in Kr a=0}
Let $ \i \beta \in K (\d M)$. Then the following assertions hold.
\item[(i)] $\beta \ge \Tm_M^{-1} $.
\item[(ii)] If $ \beta = \Tm_M^{-1} $, then $K(\d M)$ consists of a single simple quasi-eigenvalue $\i \beta$ and
$\d M = \d \Delta_{1,\Tm_M} $.
\item[(iii)] If $ \i \beta \in K_r (\d M)$, then $\beta \ge r \Tm_M^{-1} $.
\end{lem}

\begin{proof}[Proof of Theorem  \ref{t K mes}]
The inclusion $ K (\Ad ) \subset \i [ m^{-1}, + \infty ) \cup \left( \C_+ \setminus \left[ \D_{m^{-1}} ( \i m^{-1}) \cup \D_{S^{-1/2}} (0) \right] \right)  $
follows immediately from Lemma \ref{l ab in Kr a>0} (i) and Lemma \ref{l ab in Kr a=0} (i).
It is not difficult to obtain the inverse inclusion $ K (\Ad ) \supset \dots$ from Proposition \ref{p Delt} considering all the strings of type
$ \d \Delta_{x_0, A} $ that belong to $\Ad$.
\end{proof}

\begin{proof}[Proof of Theorem  \ref{t prop of optim}]
Let $\alpha \in (- S^{-1/2},S^{-1/2})$. Then Corollary \ref{c I mes} yields $\I (\alpha) > 0$.  Theorem  \ref{t K mes} implies that there exists at least one $\d M \in \Ad$
with a quasi-eigenvalue at $\k^{[\alpha]}$. In the case $\alpha \neq 0$, $\k^{[\alpha]}$ belongs to at least one of the circles $ \T_{m^{-1}} \left(\i m^{-1} \right)$,  $\T_{S^{-1/2}}
(0)$. By Lemma  \ref{l ab in Kr a>0} (i), this is possible only if one of the bounds from the definition of $\Ad$ is reached. That is, only if
one of the equalities $\Tm_{M} = m$, $\Stm_{M} = S$ hold. Now Lemma  \ref{l ab in Kr a>0} (ii)-(iii) easily implies the statement of the theorem.
The case $\alpha = 0$ can be treated similarly with the use Lemma \ref{l ab in Kr a=0}.
\end{proof}

\begin{proof}[Proof of Theorem  \ref{t sep mult}]
By Lemma \ref{l ab in Kr a>0} (iv),
$
K^{\mathrm{mult}} (\Ad) \setminus \i \R \, \subset \, \C_+ \setminus \left[  \D_{2 m^{-1}} (\i 2 m^{-1} ) \cup \D_{\sqrt{2 S^{-1}}} \, (0) \right] .
$
Lemma \ref{l ab in Kr a=0} (iii) yields $K^{\mathrm{mult}} (\Ad) \cap \i \R \subset \i [2m^{-1},+\infty)$.
This easily implies the theorem.
\end{proof}

\end{document}